\documentclass[a4paper,12pt]{article}
\usepackage{amsmath,amsthm,amsfonts,amssymb,color}
\usepackage[matrix,arrow]{xy}
\usepackage{graphicx}
\usepackage[UKenglish]{babel}
\usepackage{multicol} 
\usepackage{appendix}
\usepackage{epstopdf}

\usepackage{caption}
\usepackage{subcaption}

\usepackage{multirow} 
\usepackage{float}
\usepackage{graphicx}
\usepackage{longtable}  

\makeatletter
\renewcommand\@biblabel[1]{#1.}
\makeatother

\usepackage{amsmath,amssymb,amsthm}
\usepackage{enumerate}
\usepackage{amsfonts}
\usepackage{graphics,graphicx}
\usepackage{verbatim}
\usepackage{latexsym,makeidx}
\usepackage{amsmath,amscd}
\usepackage{appendix}
\usepackage[latin1]{inputenc}

\newenvironment{keywords}{
  \vspace{2mm}
  \noindent
  \keywordsname: 
  \itshape\small
}

\def\keywordsname{\textbf{Keywords}}
  \def\mathsubclassname{\textbf{2010 AMS Subject Classification}}

\newtheorem{theorem}{Theorem}[section]

\newcommand\be{\begin{equation}}
\newcommand\ee{\end{equation}}

 \newtheorem{thm}{Theorem}[section]
 
 \newtheorem{lemma}[thm]{Lemma}
 
 \theoremstyle{definition}
 
 \theoremstyle{remark}
 \newtheorem{remark}[thm]{Remark}
 
 \numberwithin{equation}{section}

\begin{document}

\title{A free boundary problem for a diffusion-convection equation}

\author{
Adriana C. Briozzo$^{1}$,  Domingo A. Tarzia $^{1}$\\ \\
\small {{$^1$} Depto. Matem\'atica - CONICET, FCE, Univ. Austral, Paraguay 1950} \\  
\small {S2000FZF Rosario, Argentina.}\\
\small{Email: Abriozzo@austral.edu.ar; DTarzia@austral.edu.ar.} 
}
\date{}

\maketitle



\begin{abstract}
One-dimensional free boundary problem for a nonlinear diffusion - convection equation with a Dirichlet condition at fixed face $x=0$, variable in time, is considered. Throught several transformations the problem is reduced  to a free boundary problem for a diffusion equation and the integral formulation  is obtained. By using fixed point theorems, the existence of at least a solution, for small time, to a system of coupled nonlinear integral equations is obtained.

\end{abstract}
\begin{keywords}
Diffusion-convection equation;  free boundary problem; nonlinear integral equation 
\end{keywords}

\maketitle
\section{Introduction}

Free boundary problems (FBP) are of great importance, both physically and mathematically.  FBP are boundary value problems for partial differential equations where an unknown moving boundary must be determined \cite{AlSo,Ca,Ru,Ta}. 
In this paper, we formulate a FBP for a nonlinear diffusion-convection equation namely Rosen-Fokas-Yorstos equation \cite{FoYo,Ro}. This equation describes fluid diffusion with convective effects in porous media and has multiple applications, for example, to ground water hydrology, oil reservoir engineering and other biological applications as the drug propagation in the arterial tissues.

In \cite{BrTa2019,BuDeLiFi2018} a FBP on a finite interval is formulated and solved for a nonlinear diffusion-convection equation which describe drug diffusion in arterial tissues after the drug is released by an arterial stent and the problem is reduced to a system of nonlinear integral equations.

We will study a one-dimensional FBP for the diffusion-convection equation with a variable Dirichlet condition(which is the one novelty with respect to \cite{BrTa2019,BuDeLiFi2018}) at the fixed face x=0 and a Stefan like condition on the free boundary which has a convective term.   
The present paper is organized as follows: In Section 2, we introduce the FBP  and through several transformations we map the  FBP for the nonlinear diffusion-convection equation into an equivalent FBP for the linear heat-diffusion equation. In Section 3, we give an equivalent integral formulation to problem which requires to solve a system of three coupled nonlinear Volterra integral equations. Section 4 is subdivided into two subsections: in subsection 4.1 , fixed one unknown, we prove existence and uniqueness of the solution, local in time, by using Banach fixed point theorem, in subsection 4.2 we use the Schauder fixed point theorem to prove that there exists at least one solution of this unknown. 

We can remark that sequential transformations used on Section 2 have been previously used in different physical context as modelled, in particular, by moving boundary problems, for example \cite{Br1990,BrTa1998,Fo,FoRoSc,Ro,Ro2019,RoBr,RoStCl}

\section{Free boundary problem}
We consider the free boundary $s=s(t)>0$, defined for $t>0$, and $u(x,t) $
which satisfy a diffusion-convection equation with the following conditions:
\begin{equation}
u_{t}=u^{2}(Du_{xx}-u_{x})\;\;\;,\;\;0<x<s(t)\;\;,\;\;t>0\;\;, \label{calor1}
\end{equation}
\begin{equation}
u(0,t)=f(t)\;,\;\;t>0\;\;,  \label{calor110}
\end{equation}
\begin{equation}
u(s(t),t)=\beta>0\;,\;\;t>0\;\;,  \label{tf}
\end{equation}
\begin{equation}
Du_{x}(s(t),t)-u(s(t),t)=-\dot{s}(t)\;,\;\;t>0\;\;,\label{stefan}
\end{equation}
\begin{equation}
u(x,0)=u_{0}(x)>\beta \;,\;\;0\leq x \leq b\;\;,  \label{t2}
\end{equation}
\begin{equation}
s(0)=b\;  \label{tempborde}
\end{equation}
where $D$ is the diffusivity, $u_{0}$ is the initial concentration and  $f=f(t)$ is the concentration in the fixed face $x=0$.
We assume that:
\begin{equation}
f\in C^{1}[0,\sigma],\quad u_{0}\in C^{1}[0,b],\quad u_{0}(0)=f(0),\quad u_{0}(b)=\beta, \quad f(t)>\frac{3\beta}{2}\label{hip}
\end{equation}
Following \cite{BrTa2019,BuDeLiFi2018,Ro} we will transform this problem in the one which is governed by the Burgers equation. We have:
\begin{lemma}
A) 
If $u=u(x,t)$, $s=s(t)$ is a solution to the problem (\ref{calor1})-(\ref{tempborde}) then $v=v(z,t)$, $z_{0}(t)$, $z_{1}(t)$ defined by:
\begin{equation}
v(z,t)=u(x,t),\label{firsttrans}
\end{equation}where
\begin{equation}
z(x,t)= C_{1}+\int_{0}^{t}\left(u(0,\tau)-Du_{x}(0,\tau)\right)d\tau + \int_{0}^{x}\frac{1}{u(\eta,t)}d\eta\label{zeta}
\end{equation}
\begin{equation}
z_{0}(t)=z(0,t)=C_{1}+\int^{t}_{0} \left(f(\tau)-D\frac{v_{z}(z_{0}(\tau),\tau)}{f(\tau)}\right)d\tau\label{zeta0}
\end{equation}

\begin{equation}
z_{1}(t)=z(s(t),t)=C_{2}+(\beta+1)t-\frac{D(\beta +1)}{\beta^{2}}\int^{t}_{0} v_{z}(z_{1}(\tau),\tau)d\tau\label{zeta1}
\end{equation}
with $C_{1}$ an arbitrary constant, is a solution to the problem given by the Burgers equation 
\begin{equation}
v_{t}=Dv_{zz}-2vv_{z}\;\;\;,\;\;z_{0}(t)<z<z_{1}(t)\;\;,\;\;t>0\;\;, \label{calor11}
\end{equation} 
with the following initial and boundary conditions: 

\begin{equation}
v(z_{0}(t),t)=f(t)\;,\;\;t>0\;\;,  \label{calor111}
\end{equation}
\begin{equation}
v(z_{1}(t),t)=\beta\;,\;\;t>0\;\;,  \label{tf1}
\end{equation}
\begin{equation}
D\frac{v_{z}(z_{1}(t),t)}{v(z_{1}(t),t)}-v(z_{1}(t),t)=-\frac{\beta}{\beta +1}\dot{z}_{1}(t)\;,\;\;t>0\;\;,\label{stefan1}
\end{equation}
\begin{equation}
v(z,0)=v_{0}(z)\;,\;\;C_{1}\leq z\leq C_{2}\;\;,  \label{t21}
\end{equation}
\begin{equation}
z_{0}(0)=C_{1}\;,\;\;z_{1}(0)=C_{2} \label{tempborde1}
\end{equation}

where 
\begin{equation}
v_{0}(z)=u_{0}(g^{-1}(z)),\quad\quad g(x)=C_{1}+\int^{x}_{0}\frac{1}{u_{0}(\eta)}d\eta
\end{equation}
\begin{equation}
C_{2}=C_{1}+U_{0}=C_{1}+\int^{b}_{0}\frac{1}{u_{0}(\eta)}d\eta\label{c2}
\end{equation} 
and the constants $b$, $C_{1}$ and $C_{2}$ satisfy the following relation
\begin{equation}
b=\int^{C_{2}}_{C_{1}}v_{0}(z) dz
\end{equation}
B) Conversely if $v=v(z,t)$, $z_{0}(t)$, $z_{1}(t)$ is the solution to the problem  $(\ref{calor11})-(\ref{tempborde1})$ then $u=u(x,t)$, $s=s(t)$ given by 
\begin{equation}
u(x,t)=v(z,t),\label{secondtrans}
\end{equation}
 with
\begin{equation}
x(z,t)= \int_{z_{0}(t)}^{z} v(\eta,t)d\eta,\label{equis}
\end{equation}
\begin{equation}
s(t)=x(z_{1}(t),t)=\int_{z_{0}(t)}^{z_{1}(t)} v(\eta,t)d\eta\
\label{yy}\end{equation} is a solution to the problem $(\ref{calor1})-(\ref{tempborde})$.
\end{lemma}

\begin{proof}
A) From $(\ref{firsttrans})$, $(\ref{zeta})$ and by $(\ref{calor1})$ we have 
\[
z_{x}=\tfrac{1}{u(x,t)}=\tfrac{1}{v(z,t)},\quad z_{t}=u(x,t)-Du_{x}(x,t)=v(z,t)-D\tfrac{v_{z}(z,t)}{v(z,t)},
\]
and
\[ u_{x}(x,t)=\tfrac{v_{z}(z,t)}{v(z,t)},\quad u_{xx}(x,t)=\tfrac{v_{zz}(z,t)}{v^{2}(z,t)}-\tfrac{v^{2}_{z}(z,t)}{v^{3}(z,t)}, \]

\[ u_{t}(x,t)=v_{t}(z,t)+v_{z}\left(v(z,t)-D\tfrac{v_{z}(z,t)}{v(z,t)}\right).
\]
Then, from (\ref{calor1}) we get  (\ref{calor11}) which is the Burgers equation for the dependent variable $v(z,t)$.

 Taking into account $(\ref{zeta})$  the domain $D=\left\lbrace(x,t)/0<x<s(t), t>0\right\rbrace$ for $u(x,t)$ is transformed into the domain $D^{*}=\left\lbrace(z,t)/z_{0}(t)<z<z_{1}(t),t>0)\right\rbrace$ for $v(z,t)$, where $z_{0}(t)$ and $z_{1}(t)$ are given by 
\[
z_{0}(t)=z(0,t)=C_{1}+\int_{0}^{t}\left(u(0,\tau)-Du_{x}(0,\tau)\right)d\tau 
\]

\[
z_{1}(t)=z(s(t),t)=C_{1}+\int_{0}^{t}\left(u(0,\tau)-Du_{x}(0,\tau)\right)d\tau + \int_{0}^{s(t)}\frac{1}{u(\eta,t)}d\eta\]

If we derivate $z_{1}$ respect to variable $t$ and we use $(\ref{calor1})$ and the conditions $(\ref{calor110})$-$(\ref{tempborde})$, we obtain the follow relation

\[
\dot{z}_{1}(t)=\tfrac{\beta+1}{\beta}\dot{s}_{1}(t).
\]
Then, from $(\ref{stefan})$ we have $(\ref{stefan1})$ and the expression $(\ref{zeta1})$ for $z_{1}(t)$, where $z_{1}(0)=C_{1}+\int_{0}^{b}\frac{1}{u_{0}(\eta)}d\eta=C_{2}$.
Equations $(\ref{calor111})$ and $(\ref{tf1})$ follows inmediatly from $(\ref{calor110})$ and $(\ref{tf})$ respectively.

For $t=0$ we have that
\[z=C_{1}+ \int_{0}^{x}\frac{1}{u_{0}(\eta)}d\eta=g(x),
\] 
then $(\ref{t2})$ is equivalent to $v(z,0)=u_{0}\left(g^{-1}(z)\right)$ for $C_{1}\leq z \leq C_{2}$ where $C_{2}=C_{1}+ \int_{0}^{b}\frac{1}{u_{0}(\eta)}d\eta$. Therefore $(\ref{t21})$ holds.

To prove B) we consider $(\ref{secondtrans})$, $(\ref{equis})$ and the $(\ref{calor11})-(\ref{tempborde1})$ which are satisfied by $v=v(z,t)$, $z_{0}(t)$, $z_{1}(t)$. We have
\[x_{z}=v(z,t), \quad x_{t}=Dv_{z}-v^{2}(z,t).
\]
Moreover, for $z=z_{0}(t)$ is $x=0$ and for $z=z_{1}(t)$ is $x=\int_{z_{0}(t)}^{z_{1}(t)} v(\eta,t)d\eta=s(t).$
Since
\[
v_{t}=Du^{2}_{x}u-u_{x}u^{2}+u_{t}, \quad v_{z}=u_{x}u,\quad v_{zz}=u_{xx}u^{2}+u^{2}_{x} u
\]
then $(\ref{calor11})$ yields $(\ref{calor1}).$ 

The conditions $(\ref{calor110})$, $(\ref{tf})$ and $(\ref{t2})$ follows inmediatly from $(\ref{calor111})$, $(\ref{tf1})$ and $(\ref{t21})$ respectively.

To prove $(\ref{stefan})$, from $(\ref{yy})$ we calculate $\dot{s}(t)$ and use $(\ref{calor11})$ and $(\ref{tf1})$. We have
\[
\dot{s}(t)=v(z_{1}(t),t)\dot{z}_{1}(t)-v(z_{0}(t),t)\dot{z}_{0}(t)+\int_{z_{0}(t)}^{z_{1}(t)} v_t(\eta,t)d\eta\]
\[=\beta-D\frac{v_{z}(z_{1}(t),t)}{\beta}= \beta-Du_{x}(s(t),t)
\]
and $(\ref{stefan})$ holds.

\end{proof}

\begin{remark}Eq. $(\ref{zeta})$ is equivalent to the relations
\begin{equation}
z_{x}=\frac{1}{u(x,t)},\quad\quad\quad z_{t}=u(x,t)-Du_{x}(x,t).
\end{equation}
Eq. $(\ref{equis})$ is equivalent to
\begin{equation}
x_{z}=v(z,t), \quad x_{t}=Dv_{z}-v^{2}(z,t).
\end{equation}

\end{remark} 
Now we introduce the Galilean Transformation given by
\begin{equation}
V(y,t)=v(z,t)-\beta, \quad\quad\quad y=z-2\beta t\quad\quad t>0\label{2}
\end{equation}
to obtain de following result:

\begin{lemma} Under the transformation $(\ref{2})$ the problem (\ref{calor11})-(\ref{c2}) is equivalent to the following FBP:
\begin{equation}
V_{t}=DV_{yy}-2VV_{y}\;\;\;,\;\;y_{0}(t)<y<y_{1}(t)\;\;,\;\;t>0\;\;, \label{cal}
\end{equation}
\begin{equation}
V(y_{0}(t),t)=f(t)-\beta\;,\;\;t>0\;\;,  \label{cal1}
\end{equation}
\begin{equation}
V(y_{1}(t),t)=0\;,\;\;t>0\;\;,  \label{t1}
\end{equation}
\begin{equation}
D\frac{V_{y}(y_{1}(t),t)}{\beta}=\frac{\beta(1-\beta)-\beta\dot{y}_{1}(t)}{\beta +1}\;,\;\;t>0\;\;,\label{ste1}
\end{equation}
\begin{equation}
V(y,0)=V_{0}(y)\;,\;\;C_{1}\leq y \leq C_{2}\;\;,  \label{tt21}
\end{equation}
\begin{equation}
y_{0}(0)=C_{1}\;,\;\;y_{1}(0)=C_{2} \label{temp}
\end{equation}
where
 \begin{equation}
V_{0}(y)=v_{0}(y)-\beta
\end{equation}

\begin{equation}
y_{0}(t)=C_{1}-2\beta t+\int^{t}_{0} \left(f(\tau)-D\frac{V_{y}(y_{0}(\tau),\tau)}{f(\tau)}\right)d\tau\label{f}
\end{equation}

\begin{equation}
y_{1}(t)=C_{2}+(1-\beta)t-\frac{D(\beta +1)}{\beta^{2}}\int^{t}_{0} V_{y}(y_{1}(\tau),\tau)d\tau.\label{fr}
\end{equation}
\end{lemma}
\begin{proof}
The Galilean transformation $(\ref{2})$ leaves invariant the Burgers equation $(\ref{calor11})$. The free boundaries $y_{0}(t)$ and $y_{1}(t)$ given by $(\ref{f})$-$(\ref{fr})$ are obtained from $(\ref{zeta0})$-$(\ref{zeta1})$. The conditions $(\ref{cal1})$-$(\ref{temp})$ follows from $(\ref{calor111})$-$(\ref{tempborde1})$.

Conversely, if we define
\[v(z,t)=V(y,t)+\beta, \quad\quad\quad z=y+2\beta t\quad\quad t>0,
\] from (\ref{cal})-(\ref{fr}) we obtain (\ref{calor11})-(\ref{c2}) with $z_{0}(t)$ and $z_{1}(t)$ given by $(\ref{zeta0})$ and $(\ref{zeta1})$ respectively.
\end{proof}

Let us now transform problem $(\ref{cal})-(\ref{fr})$ in the one which is governed by a heat-diffusion equation using the Hopf Cole transformation given by
\begin{equation}
w(y,t)=C(t)V(y,t)\eta(y,t), \quad y_{0}(t)\leq y \leq y_{1}(t)\;\;,\;\;t>0,\label{tres}
\end{equation}with
\begin{equation}
C(t)=1-\int^{t}_{0} w_{y}(y_{1}(\tau),\tau)d\tau,\label{ce}
\end{equation}and
\begin{equation}
\eta(y,t)=exp\left(\tfrac{1}{D}\int_{y}^{y_{1}(t)}V(\xi,t)d\xi\right).\label{eta}
\end{equation}

We have the following result:
\begin{theorem}
Under transformation $(\ref{tres})-(\ref{eta})$ problem $(\ref{cal})-(\ref{fr})$ is equivalent to the free boundary problem $(\ref{calu})-(\ref{free})$ given by:
\begin{equation}
w_{t}=Dw_{yy}\;\;\;,\;\;y_{0}(t)<y<y_{1}(t)\;\;,\;\;t>0\;\;, \label{calu}
\end{equation}
\begin{equation}
w(y_{0}(t),t)=(f(t)-\beta)\left(C(t)+\frac{1}{D}\int_{y_{0}(t)}^{y_{1}(t)}w(\xi,t)d\xi\right)\;,\;\;t>0\;\;,  \label{cal1.}
\end{equation}
\begin{equation}
w(y_{1}(t),t)=0\;,\;\;t>0\;\;,  \label{t1.}
\end{equation}
\begin{equation}
\frac{Dw_{y}(y_{1}(t),t)}{\beta C(t)}=\frac{\beta(1-\beta)-\beta\dot{y}_{1}(t)}{\beta +1}\;,\;\;t>0\;\;,\label{ste1.}
\end{equation}
\begin{equation}
w(y,0)=F(y)\;,\;\;C_{1}\leq y \leq C_{2}\;\;,  \label{tt21.}
\end{equation}
\begin{equation}
y_{0}(0)=C_{1}\;,\;\;y_{1}(0)=C_{2} \label{tempu.}
\end{equation}
where 
\begin{equation}
F(y)=V_{0}(y)exp\left(\tfrac{1}{D}\int_{y}^{C_{2}}V_{0}(\xi)d\xi\right)=V_{0}(y)\left(1-\frac{1}{D}\int_{y}^{C_{2}}w(\xi,0)d\xi\right)\label{efemay}
\end{equation}
and the free boundaries $y_{0}=y_{0}(t)$ and $y_{1}=y_{1}(t)$ are given by:
\begin{equation}
y_{0}(t)=C_{1}-\beta^{2}\int^{t}_{0} \frac{1}{f(\tau)}d\tau-D\int^{t}_{0}\frac{w_{y}(y_{0}(\tau),\tau)}{w(y_{0}(\tau),\tau)}\left(1-\tfrac{\beta}{f(\tau)}\right)d\tau,\label{freee}
\end{equation}
\begin{equation}
y_{1}(t)=C_{2}+(1-\beta)t+\frac{D(\beta +1)}{\beta^{2}}log\left(1-\int^{t}_{0} w_{y}(y_{1}(\tau),\tau)d\tau\right).\label{free}
\end{equation}
\end{theorem}
\begin{proof} To prove the equivalence of the two problems we will deduce the inverse transformation to the relation$(\ref{tres})$ by considering the definition $(\ref{eta})$, we have
\[
log\left(\eta(y,t)\right)=\tfrac{1}{D}\int_{y}^{y_{1}(t)}V(\xi,t)d\xi
\] then
\[
\eta_{y}(y,t)=-\tfrac{1}{D}V(y,t)\eta(y,t)=-\tfrac{1}{D}\frac{w(y,t)}{C(t)}
.\]
Integrating on variable $y$, it follows that
\[
\eta(y,t)=\frac{C(t)+\frac{1}{D}\int_{y}^{y_{1}(t)}w(\xi,t)d\xi}{C(t)}.
\]
Therefore, we have that the inverse relation to the generalized Hopf-Cole transformation  $(\ref{tres})$ is expressed by: 
\begin{equation}
V(y,t)=\frac{w(y,t)}{C(t)+\frac{1}{D}\int_{y}^{y_{1}(t)}w(\xi,t)d\xi}.\label{inverse}
\end{equation}
Under transformation $(\ref{inverse})$ the Burgers equation $(\ref{cal})$ is mapped into the linear heat-diffusion equation $(\ref{calu}).$ The initial and boundary conditions $(\ref{cal1.})-(\ref{tempu.})$ are easily obtained from $(\ref{cal1})-(\ref{temp})$. The expressions $(\ref{freee})$ and $(\ref{free})$ for the free boundaries are obtained from $(\ref{f})$ and $(\ref{fr})$ respectively.

The converse is proved analogously.
\end{proof}

\section{Integral formulation}
In this section, we give an integral formulation of the free boundary problem $(\ref{calu})-(\ref{free})$. We have the following equivalence theorem. 
\begin{theorem}
 Let $(\ref{hip})$ and $0<D<2$ be. The solution to the free boundary problem $(\ref{calu})-(\ref{free})$ has the following integral representation 
\begin{equation}
w(y,t)=\int\nolimits_{C_{1}}^{C_{2}}G(y,t;\xi ,0)F(\xi
)d\xi +D
\int_{0}^{t}\phi_{1}(\tau )G(y,t;y_{1}(\tau ),\tau )d\tau  \label{z}
\end{equation}
\[
+\beta^{2}\int_{0}^{t} \frac{h(\tau)}{f(\tau)} G(y,t;y_{0}(\tau ),\tau )d\tau
-D\beta\int_{0}^{t} \frac{\phi_{2}(\tau)}{f(\tau)} G(y,t;y_{0}(\tau ),\tau ) d\tau 
\]
\[-D\int_{0}^{t} h(\tau) N_{y}(y,t;y_{0}(\tau ),\tau ) d\tau, 
\]
with
\begin{equation}
h(t)=(f(t)-\beta)\left(C(t)+\frac{1}{D}\int_{y_{0}(t)}^{y_{1}(t)}w(\xi,t)d\xi\right),\label{ache}
\end{equation}
\begin{equation}
y_{0}(t)=C_{1}-\beta^{2}\int^{t}_{0} \frac{1}{f(\tau)}d\tau-D\int^{t}_{0}\tfrac{\phi_{2}(\tau)}{h(\tau)}\left(1-\tfrac{\beta}{f(\tau)}\right)d\tau,\label{ycero}
\end{equation}
\begin{equation}
y_{1}(t)=C_{2}+(1-\beta)t+\tfrac{D(\beta +1)}{\beta^{2}}ln\left(1-\int^{t}_{0} \phi_{1}(\tau)d\tau\right) \label{ese}
\end{equation}
and $\phi_{1}$, $\phi_{2}$ are defined by 
\begin{equation}
\phi_{1}\left(t\right) =\frac{\partial w
}{\partial y}\left( y_{1}(t),t\right) \;\;,\;\;\phi_{2}\left( t\right) =
\frac{\partial w }{\partial y}\left(
y_{0}(t),t\right) \label{def}
\end{equation}
if and only if it satisfies the following system of two Volterra integral
equations: 
\[
\phi_{1}\left(t\right)
 =\frac{2}{2-D }\left\{
\int\nolimits_{C_{1}}^{C_{2}}N(y_{1}(t),t;\xi ,0)F^{\prime }(\xi )d\xi+
D
\int_{0}^{t}\phi_{1}(\tau )G_{y}(y_{1}(t),t;y_{1}(\tau ),\tau )d\tau \right. 
\]
\[
+\beta^{2}\int_{0}^{t} \frac{h(\tau)}{f(\tau)} G_{y}(y_{1}(\tau),t;y_{0}(\tau ),\tau )d\tau
-D\beta\int_{0}^{t} \frac{\phi_{2}(\tau)}{f(\tau)} G_{y}(y_{1}(\tau),t;y_{0}(\tau ),\tau ) d\tau ,
\]
\begin{equation}\left.-\int_{0}^{t} h'(\tau) N(y_{1}(t),t;y_{0}(\tau ),\tau ) d\tau 
\right\rbrace,
\label{ecintegralf}
\end{equation}
\[
\phi_{2}\left(t\right) =\frac{2f(t)}{2f(t)-D\beta}\left\{-\beta^{2}\frac{h(t)}{f(t)}
+\int\nolimits_{C_{1}}^{C_{2}}N(y_{0}(t),t;\xi ,0)F^{\prime }(\xi )d\xi
\right.  
\]
\[
 + D
\int_{0}^{t}G_{y} (y_{0}(t),t;y_{1}(\tau ),\tau )\phi_{1}(\tau )d\tau
+\beta^{2}\int_{0}^{t}\frac{h(\tau)}{f(\tau)}G_{y}(y_{0}(t),t;y_{0}(\tau ),\tau )(\tau )d\tau 
\]
\begin{equation}
\left. - D\beta
\int_{0}^{t}\frac{\phi_{2}(\tau)}{f(\tau)}G_{y}(y_{0}(t),t;y_{0}(\tau ),\tau )(\tau )d\tau 
-\int_{0}^{t}h'(\tau)N(y_{0}(t),t;y_{0}(\tau ),\tau )(\tau )d\tau \right\}, \label{ecintegralf1}
\end{equation}
where $G$, $N$ are the Green and Neumann functions respectively, and $K\;$is the
fundamental solution to the heat equation, defined by 
\begin{equation}
G\left( x,t,\xi ,\tau \right) =K\left( x,t,\xi ,\tau \right) -K\left(
-x,t,\xi ,\tau \right),  \label{defG}
\end{equation}
\begin{equation}
N\left( x,t,\xi ,\tau \right) =K\left( x,t,\xi ,\tau \right) +K\left(
-x,t,\xi ,\tau \right),  \label{defN}
\end{equation}
\begin{equation}
K\left( x,t,\xi ,\tau \right) =\left\{ 
\begin{array}{ll}
\frac{1}{2\sqrt{\pi D\left( t-\tau \right) }}\exp \left( -\frac{\left( x-\xi
\right) ^{2}}{4D\left( t-\tau \right) }\right) & t>\tau \\ 
0 & t\leq \tau
\end{array}
\right.  \label{defK}
\end{equation}
and $y_{0}\;$, $y_{1}$ are given by $\left( \ref{ycero}\right) $ and $%
\left( \ref{ese}\right) $ respectively. Moreover, function $h(t)=w(y_{0}(t),t)$ must satisfy the integral relation  
\begin{equation}
h(t)= (f(t)-\beta)\left(1-\int_{0}^{t}\phi_{1}(\tau)d\tau+\frac{1}{D}\int_{y_{0}(t)}^{y_{1}(t)}w(y,t)dy\right).\label{eqache}
\end{equation}
\end{theorem} 
\begin{proof}
Let $w (y,t)$, $y_0 (t)$, $y_1(t)$ be the solution to the problem $(\ref{calu})-(\ref{free})$. We integrate on the domain 
\[
D_{t,\epsilon}=\left\{ \left( \xi ,\tau \right) \text{ }/\text{ }%
y_{0}(\tau )<\xi <y_{1}\left( \tau \right) ,\text{ }\epsilon <\tau
<t-\epsilon \right\} \;(\epsilon >0), 
\]
$\;$ the Green identity 
\begin{equation}
D\left( Gw _{\xi}-wG_{\xi }\right) _{\xi }-\left( Gw
\right) _{\tau }=0\;\;  \label{ngreen}
\end{equation}
and we let $\epsilon \rightarrow 0$, to obtain the integral
representation for $w (y,t)\;$\cite{Fr1959,Ru} 
\begin{equation}
w(y,t)=\int\nolimits_{C_{1}}^{C_{2}}G(y,t;\xi ,0)w(\xi,0)
d\xi +D
\int_{0}^{t}w_{\xi}(y_{1}(\tau),\tau)G(y,t;y_{1}(\tau ),\tau )d\tau  \label{now}
\end{equation}
\[
+\beta^{2}\int_{0}^{t} \frac{w(y_{0}(\tau),\tau)}{f(\tau)} G(y,t;y_{0}(\tau ),\tau )d\tau
-D\beta\int_{0}^{t} \frac{w_{\xi}(y_{0}(\tau),\tau)}{f(\tau)} G(y,t;y_{0}(\tau ),\tau ) d\tau 
\]
\[+D\int_{0}^{t} w(y_{0}(\tau),\tau) G_{\xi}(y,t;y_{0}(\tau ),\tau ) d\tau .
\]
By using the definitions of $\phi_{1}$ and $\phi_{2}$ given by $(\ref{def})$, the definition of $h$ and boundary conditions we have $(\ref{z}).$
If we differentiate (\ref{now}) in variable $y$ and we let $%
y\rightarrow y_{0}^{+}(t)$ and $y\rightarrow y_{1}^{-}(t),$ by using
the jump relations \cite{Fr1959} we obtain the system of integral equations $%
\left( \ref{ecintegralf}\right) $ and $\left( \ref{ecintegralf1}\right) $
for $\phi_{1}$ and $\phi_{2}.$ Moreover, from (\ref{ce}) and (\ref{ache}) we have the equation (\ref{eqache}).

Conversely, the function $w(y,t)$ defined by (\ref{z}), where $%
\phi_{1} $ and $\phi_{2}$ are the solutions of $\left( \ref{ecintegralf}\right) $ and $%
\left( \ref{ecintegralf1}\right) ,$ satisfies the conditions $(\ref{calu})$, $(\ref{ste1.})$ - $(\ref{tempu.})$. In order to
prove the conditions $(\ref{cal1.})$ and $(\ref{t1.})$ we define 
\[
\mu _{1}\left( t\right) =w(y_{1}(t),t)\;\text{%
and}\;\mu _{2}\left( t\right) =h(t)-w(y_{0}(t),t).\; 
\]
\\
If we integrate the Green identity\ (\ref{ngreen})$\;$over the domain $%
D_{t,\varepsilon }$ $\left( \varepsilon >0\right) $ and we let $\varepsilon
\rightarrow 0,$ we obtain that 
\[
w(y,t)=\int\nolimits_{C_{1}}^{C_{2}}G(y,t;\xi
,0)w(\xi,0)d\xi +D\int_{0}^{t} G(y,t;y_{1}(\tau ),\tau )\phi_{1}(\tau)d\tau 
\]
\[-D\int_{0}^{t} G_{y}(y,t;y_{1}(\tau ),\tau )w
(y_{1}(\tau ),\tau )d\tau +\int_{0}^{t} G(y,t;y_{1}(\tau ),\tau )w(y_{1}(\tau),\tau)y^{'}_{1}(\tau)d\tau 
\]
\[
-\int_{0}^{t}G(y,t;y_{0}(\tau ),\tau )\left[w(y_{0}(\tau),\tau)y^{'}_{0}(\tau) -D\phi_{2}(\tau)\right]d\tau
\]
\begin{equation}
+D\int_{0}^{t} G_{\xi}(y,t;y_{0}(\tau ),\tau )w(y_{0}(\tau),\tau) d\tau. \label{zbis}
\end{equation}
\\
Then, if we compare this last expression (\ref{zbis}) with (\ref{z}) we
deduce that 
\[
\int_{0}^{t}G(y,t;y_{0}(\tau ),\tau )\left[\frac{\beta^{2}}{f(\tau)}\mu_{2}(\tau)-\frac{D\phi_{2}(\tau)}{f(\tau)}\left(\beta+\frac{w(y_{0}(\tau),\tau)(f(\tau)-\beta)}{h(\tau)}\right)+D\phi_{2}(\tau)\right]d\tau
\]
\[
+D\int_{0}^{t} G_{y}(y,t;y_{0}(\tau ),\tau )\mu_{2}(\tau) d\tau+D\int_{0}^{t} G_{y}(y,t;y_{1}(\tau ),\tau )\mu_{1}(\tau)d\tau 
\]
\begin{equation}
-\int_{0}^{t} G(y,t;y_{1}(\tau ),\tau )\mu_{1}(\tau)\left[(1-\beta)-\frac{D\phi_{1}(\tau)(\beta+1}{\beta^{2}C(\tau)}\right] d\tau =0.
\label{des}
\end{equation}
\\
By taking $y\rightarrow y_{1}^{-}(t)$ and $y\rightarrow y_{0}^{+}(t)$
in (\ref{des}), and the jump relations we obtain that $\mu _{1}$
and $\mu _{2}\;$ must satisfy the following system of Volterra integral
equations: 
\begin{equation}
\mu _{1}(t)=\frac{-2}{D}\int_{0}^{t} \left\lbrace DG_{y}(y_{1}(t),t;y_{1}(\tau ),\tau )-G(y,t;y_{1}(\tau ),\tau )\left[(1-\beta)-\tfrac{D\phi_{1}(\tau)(\beta+1}{\beta^{2}C(\tau)}\right]\right\rbrace\mu_{1}(\tau),  \label{fiuno}
\end{equation}
\[
+ \left\lbrace DG_{y }(y_{1}(t),t;y_{0}(\tau ),\tau
)+G(y_{1}(t),t;y_{0}(\tau ),\tau )\left[\tfrac
{\beta}{f(\tau)}-D\phi_{2}(\tau)\left(\tfrac
{\beta}{f(\tau}-\tfrac{1}{h(\tau)}\right)\right]\right\rbrace\mu_{2}(\tau)d\tau
\]
\\
\begin{equation}
\mu _{2}(t)=\frac{2}{D}+\int_{0}^{t} \left\lbrace DG_{y}(y_{0}(t),t;y_{0}(\tau ),\tau )-G(y_{0}(t),t;y_{0}(\tau ),\tau )\right.\label{fidos}
\end{equation}
\[
\left.\left[\tfrac
{\beta}{f(\tau)}-D\phi_{2}(\tau)\left(\tfrac
{\beta}{f(\tau}-\tfrac{1}{h(\tau)}\right)\right]\right\rbrace \mu_{2}(\tau)  
\]
\[
+\left\lbrace D G_{y}(y_{0}(t),t;y_{1}(\tau ),\tau )- G(y_{0}(t),t;y_{1}(\tau ),\tau )\left[(1-\beta)-\frac{D\phi_{1}(\tau)(\beta+1}{\beta^{2}C(\tau)}\right]\right\rbrace\mu_{1}(\tau) d\tau .
\]
\\
Following \cite{Mi}, it's easy to see that there exist a unique solution $%
\mu _{1}\equiv \mu _{2}\equiv 0$ to the system of Volterra integral
equations (\ref{fiuno})-(\ref{fidos}). Then $(\ref{cal1.})$ and $(\ref{t1.})$
are verified and the result holds.\medskip
\end{proof}

\begin{section}{Existence of the solution}

In order to prove existence of solution $w=w(y,t)$, $y=y_{0}(t)$ and $y=y_{1}(t)$ of $(\ref{calu})-(\ref{free})$ and taking into account the result of Theorem 3.1 we will demonstrate that there exists at least a local solution $\phi_{1}$,$\phi_{2}$ and $h$ to the coupled nonlinear integral equations $\left( \ref{ecintegralf}\right) $, $%
\left( \ref{ecintegralf1}\right)$ and $(\ref{eqache})$.
\\
We will proceed in the following way:  
Fixed positive constants $H$,$R$, $S$ and $\sigma$ we define the set $\Pi=\Pi(H,R,S,\sigma)$ given by 
\begin{equation}
\Pi :=\left\lbrace h\in C^{1}[0,\sigma]/h(t)\geq H , \left\|h\right\|\leq R,\left\|h'\right\|\leq S\right\rbrace
\end{equation} where  $\left\|h\right\|=\max \limits_{
t\in[0,\sigma]}\left|h(t)\right|$. Clearly $\Pi$ is a compact and convex set in $C^{1}[0,\sigma]$. 
\\
 For each fixed function $h\in\Pi_1=\left\lbrace h\in C^{1}[0,1]/h(t)\geq H , \left\|h\right\|\leq R,\left\|h'\right\|\leq S\right\rbrace$ we will use the Banach fixed point Theorem in order to prove that  there exist unique solutions $\phi_{1},\;\phi_{2}\;\in C^{0}\left[ 0,\sigma
\right] \;$ to the system of two Volterra integral equations (\ref
{ecintegralf})\ and (\ref{ecintegralf1}). Then for suitable $H,R,S$ and $\sigma$, by using Shauder's fixed point Theorem we will demonstrate that there exists at least a solution $h \in \Pi_{1}$ of (\ref{eqache}).
\begin{subsection}{Existence and uniqueness of $\phi_{1},\;\phi_{2}$}
We consider the Banach space \[
\textbf{C}[0,\sigma]=\left\{ \stackrel{\longrightarrow }{\phi^{*}}=\binom{\phi_{1}}{\phi_{2}}%
/\;\phi_{i}:\left[ 0,\sigma \right] \rightarrow {\Bbb R} ,\quad i=1,2,\;\text{continuous}
\right\} 
\] 
with the norm
\[
\left\| \stackrel{\longrightarrow }{\phi^{*}}\right\| _{\sigma }:=\max \limits_{
t\in \left[ 0,\sigma \right] }\left| \phi_{1}(t)\right| +\max\limits_{
t\in \left[ 0,\sigma \right] }\left| \phi_{2}(t)\right| 
\] and the subset:
\[
C_{M,\sigma}=\left\{ \stackrel{\longrightarrow }{\phi^{*}}\in \textbf{C}[0,\sigma] /\left\| \stackrel{\longrightarrow }{\phi^{*}}\right\|_{\sigma }\leq
M\right\} 
\]
with $\sigma \,\,$ and $M$ positive numbers to be determinate. 
\\
We define the map $\chi:C_{M,\sigma }\longrightarrow C_{M,\sigma },$ such that 
\[
\chi\left( \stackrel{%
\longrightarrow }{\phi^{*}}\right)(t) =\binom{\chi_{1}(\phi_{1}(t),\phi_{2}(t))}{\chi_{2}(\phi_{1}(t),\phi_{2}(t)))%
} 
\]
where 
\[
\chi_{1}(\phi_{1}(t),\phi_{2}(t))=\tfrac{2}{2-D }\left\{
\int\nolimits_{C_{1}}^{C_{2}}N(y_{1}(t),t;\xi ,0)F^{\prime }(\xi )d\xi+
D
\int_{0}^{t}\phi_{1}(\tau )G_{y}(y_{1}(t),t;y_{1}(\tau ),\tau )d\tau \right. 
\]
\[
+\beta^{2}\int_{0}^{t} \frac{h(\tau)}{f(\tau)} G_{y}(y_{1}(\tau),t;y_{0}(\tau ),\tau )d\tau
-D\beta\int_{0}^{t} \frac{\phi_{2}(\tau)}{f(\tau)} G_{y}(y_{1}(\tau),t;y_{0}(\tau ),\tau ) d\tau ,
\]
\begin{equation}\left.-\int_{0}^{t} h'(\tau) N(y_{1}(t),t;y_{0}(\tau ),\tau ) d\tau 
\right\rbrace,
\label{F1}
\end{equation}

\[
\chi_{2}(\phi_{1}(t),\phi_{2}(t))=
\frac{2f(t)}{2f(t)-D\beta}\left\{-\beta^{2}\frac{h(t)}{f(t)}
+\int\nolimits_{C_{1}}^{C_{2}}N(y_{0}(t),t;\xi ,0)F^{\prime }(\xi )d\xi
\right.  
\]
\[
 + D
\int_{0}^{t}G_{y} (y_{0}(t),t;y_{1}(\tau ),\tau )\phi_{1}(\tau )d\tau
+\beta^{2}\int_{0}^{t}\frac{h(\tau)}{f(\tau)}G_{y}(y_{0}(t),t;y_{0}(\tau ),\tau )(\tau )d\tau 
\]
\begin{equation}
\left. - D\beta
\int_{0}^{t}\frac{\phi_{2}(\tau)}{f(\tau)}G_{y}(y_{0}(t),t;y_{0}(\tau ),\tau )(\tau )d\tau 
-\int_{0}^{t}h'(\tau)N(y_{0}(t),t;y_{0}(\tau ),\tau )(\tau )d\tau \right\}. 
\label{F2}
\end{equation}
We will prove that for suitable $M$ and $\sigma$, the map $\chi$ is well defined and it is also a contraction, therefore by the Banach fixed point Theorem it has a unique fixed point. 

Firstly, we give some preliminary results
\begin{lemma}\label{cotasy}
Let $f(t)>\frac{3\beta}{2}$, $0<D<2$ and $\phi_{i}\in C^{0}\left[ 0,\sigma \right] ,\max \limits_ {t\in \left[ 0,\sigma
\right]}\left| \phi_{i}(t)\right| \leq M,(i=1,2)$. If $$2(1+\beta)\left(1+\frac{M}{\beta^{2}}\right)\sigma \leq C_{2},\quad\quad 2\left(\beta+2\frac{MD}{H}\right)\sigma\leq C_{1}\;$$ then $y_{0}$ and $y_{1}$ defined by (\ref{ycero})$\;$%
and (\ref{ese}) satisfies 
\begin{equation}
\left| y_{0}(t)-y_{0}(\tau) \right| \leq \left(\beta+2\tfrac{DM}{H}\right)\left| t-\tau
\right| \text{ \ ,\ }\forall \tau ,t\in \left[ 0,\sigma \right], \label{io}
\end{equation}
\begin{equation}
\tfrac{C_{1}}{2}\leq y_{0}(t) \leq 3\tfrac{C_{1}}{2},\forall
t\in \left[ 0,\sigma \right], \label{io1}
\end{equation}
\begin{equation}
\left| y_{1}(t)-y_{1}(\tau) \right| \leq (1+\beta)\left(1+\tfrac{M}{\beta^{2}}\right)\left|t-\tau
\right| \text{ \ ,\ }\forall \tau ,t\in \left[ 0,\sigma \right],\label{io2}
\end{equation}
\begin{equation}
\tfrac{C_{2}}{2}\leq y_{1}(t)\leq 3\tfrac{C_{2}}{2},\text{ }%
\forall t\in \left[ 0,\sigma \right].\label{io3}
\end{equation}
\end{lemma}
\begin{proof} It follows inmediatly from definitions (\ref{ycero})-(\ref{ese}) and assumptions on data.
\end{proof}
To prove the following Lemmas we need to use the classical inequality 
\begin{equation}
\dfrac{\exp \left( \frac{-x^{2}}{\alpha \left( t-\tau \right) }\right) }{%
\left( t-\tau \right) ^{\frac{n}{2}}}\leq \left( \frac{n\alpha }{2ex^{2}}%
\right) ^{^{\frac{n}{2}}}\;,\;\alpha ,x>0\;,\;t>\tau \;,\;n\in {\Bbb N.}
\label{exp}
\end{equation}

\begin{lemma}\label{cotasint}
$\;$Let $\sigma \leq 1.$ For each function $h\in \Pi_{1}$ under the hypothesis of Lemma $\ref{cotasy}$ and $C_{1}<\frac{U_{0}}{2}$ we have that following properties are satisfied
\begin{equation}
\int\nolimits_{C_{1}}^{C_{2}}\left| F^{\prime }(\xi )\right| \left|
N(y_{1}(t),t;\xi ,0)\right| d\xi \leq \left\| F^{\prime
}\right\| \leq A_{1}(u_{0},U_{0},\beta, D), \label{i}
\end{equation}
\begin{equation}
D\int_{0}^{t}\left| G_{y}(y_{1}(t),t;y_{1}(\tau ),\tau )\phi_{1}(\tau
)\right| d\tau \leq A_{2}(M,D,U_{0},C_{1})\;\sqrt{\sigma},  \label{olvidada}
\end{equation}
\begin{equation}
\beta^{2}\int_{0}^{t}\left| G_{y}(y_{1}(t),t;y_{0}(\tau ),\tau )\frac{h(\tau
)}{f(\tau)}\right| d\tau \leq A_{3}(R,D,\beta,C_{2},C_{1})\;\sqrt{\sigma}  ,\label{olv}
\end{equation}
\begin{equation}
\beta D\int_{0}^{t}\left| G_{y}(y_{1}(t),t;y_{0}(\tau ),\tau )\frac{\phi_{2}(\tau
)}{f(\tau)}\right| d\tau \leq A_{4}(M,D,C_{2},C_{1})\;\sqrt{\sigma},\label{olv1}
\end{equation}
\medskip 
\begin{equation}
\int_{0}^{t}\left| h'(\tau )\right| \left|
N(y_{1}(t),t;y_{0}(\tau ),\tau )\right| d\tau \leq A_{5}(S,D)\sqrt{\sigma} ,\label{ii}
\end{equation}
\medskip 
\begin{equation}
\int\nolimits_{C_{1}}^{C_{2}}\left| F^{\prime }(\xi )\right| \left|
N(y_{0}(t),t;\xi ,0)\right| d\xi \leq \left\| F'
\right\| \leq A_{1}(u_{0},U_{0},\beta, D),  \label{iii}
\end{equation}
\medskip 
\begin{equation}
D\int_{0}^{t}\left| G_{y}(y_{0}(t),t;y_{1}(\tau ),\tau )\phi_{1}(\tau
)\right| d\tau \leq A_{4}(D,M,C_{1},C_{2})\sqrt{\sigma}, \label{olvidada1}
\end{equation}
\begin{equation}
\beta^{2}\int_{0}^{t}\left| G_{y}(y_{0}(t),t;y_{0}(\tau ),\tau )\frac{h(\tau
)}{f(\tau)}\right| d\tau \leq A_{6}(R,M,D,\beta,C_{2},C_{1})\;\sqrt{\sigma}  ,\label{olv2}
\end{equation}
\begin{equation}
\beta D\int_{0}^{t}\left| G_{y}(y_{0}(t),t;y_{0}(\tau ),\tau )\frac{\phi_{2}(\tau
)}{f(\tau)}\right| d\tau \leq A_{7}(M,D,C_{1})\;\sqrt{\sigma},\label{delhoy}
\end{equation}
\medskip 
\begin{equation}
\int_{0}^{t}\left| h'(\tau )\right| \left|
N(y_{0}(t),t;y_{0}(\tau ),\tau )\right| d\tau \leq A_{5}(S,D)\sqrt{\sigma}  ,\label{iibis}
\end{equation}
\begin{equation}
\beta^{2}\frac{h(t)}{f(t)}\leq \beta R, 
\end{equation}
\medskip where 
\[
A_{1}(u_{0},U_{0},\beta, D)=exp\left(\tfrac{\left\|u_{0}\right\|+\beta U_{0}}{D}\right) \left[\left\|\frac{u_{0}^{'}}{u_{0}}\right\|+\frac{\left\|u_{0}\right\|+\beta}{D} \right],
\]
\[
A_{2}(M,D,U_{0},C_{1})=\frac{M\sqrt{D}}{2\sqrt{\pi }}\left[ 2M+\frac{3}{C_{2}^{2}}\left( \frac{2D%
}{3e}\right) ^{3/2}\right] ,
\]
\[
A_{3}(R,D,\beta,C_{2},C_{1})=\frac{R\beta}{2\sqrt{D\pi}}(A_{31}+A_{32)}),
\]
\[A_{31}=\frac{3C_{2}-C_{1}}{2}\left(\frac{24D}{e(C_{2}-3C_{1})^{2}}\right)^{\frac{3}{2}},\quad\quad A_{32}=\frac{18\sqrt{6}}{e^{3/2}(C_{1}+C_{2})^{2}},
\]
\[
A_{4}(M,D,C_{2},C_{1})=\frac{M\sqrt{D}}{2\sqrt{\pi }}(A_{31}+A_{32)}),
\]
\[
A_{5}(S,D)=2{\frac{S}{%
\sqrt{D\pi }}},
\]
\[
A_{6}(R,M,D,\beta,C_{1})=\frac{\beta R M\sqrt{D}}{2\sqrt{\pi }}\left[ 2M+\frac{3}{C_{1}^{2}}\left( \frac{2D%
}{3e}\right) ^{3/2}\right], 
\]\[
A_{7}(M,D,\beta,C_{1}))=DM\left[ 2M+\frac{3}{C_{1}^{2}}\left( \frac{2D%
}{3e}\right) ^{3/2}\right] ,
\]

\medskip 
 \end{lemma}

\begin{proof}

To prove (\ref{i}) we consider 
\[
\int\nolimits_{C_{1}}^{C_{2}}\left| F^{\prime }(\xi )\right| \left|
N(y_{1}(t),t;\xi ,0)\right| d\xi \leq \left\| F^{\prime
}\right\| \int_{0}^{\infty }\left| N(y_{1}(t),t;\xi ,0)\right|
d\xi \leq \left\| F^{\prime}\right\| .
\]
From $(\ref{efemay})$ we have
\[
F'(y)=exp\left(\tfrac{1}{D}\int_{y}^{C_{2}}V_{0}(\xi)d\xi\right) \left[V^{'}_{0}(y)-\frac{1}{D}V_{0}^{2}(y)\right]
\] then
\[
\left\|F'\right\|\leq exp\left(\tfrac{\left\|V_{0}\right\|(C_{2}-C_{1})}{D}\right) \left[\left\|V_{0}^{'}\right\|+\tfrac{1}{D}\left\|V_{0}\right\|\right]\]
\[\leq exp\left(\tfrac{(\left\|u_{0}\right\|+\beta) U_{0}}{D}\right) \left[\left\|\tfrac{u_{0}^{'}}{u_{0}}\right\|+\tfrac{\left\|u_{0}\right\|+\beta}{D} \right]=A_{1}(u_{0},U_{0},\beta, D).
\]
Following the proof given in \cite{BrTa2006,BrNa2012} and taking $C_{1}<\frac{U_{0}}{2}$ we obtain $(\ref{olvidada})$, $(\ref{olv})$, $(\ref{olv1})$,$(\ref{olvidada1})$, $(\ref{olv2})$ and $(\ref{delhoy}).$ 

To prove (\ref{ii}) we take into account that 
\[
\left| N(y_{1}(t),t;y_{0}(\tau ),\tau )\right| \leq \frac{1}{%
\sqrt{\pi \left( t-\tau \right) }} 
\]
so, we obtain 
\[
\int_{0}^{t}\left| {h^{'}}(\tau )\right| \left|
N(y_{1}(t,t;y_{0}(\tau ),\tau )\right| d\tau \leq 2\sqrt{%
\frac{t}{D\pi }} S. 
\]

The inequalities (\ref{iii}) and (\ref{iibis}) are proved in the same way as (\ref{i}) and (\ref{ii}) respectively.
\end{proof}
\begin{lemma}\label{cotasyi}
Let $y_{01}$ and $y_{02}$ be the functions corresponding to $\phi_{21}$
and $\phi_{22}$ in $C^{0}[0,\sigma ]$ respectively, and $y_{11}$ and $%
y_{12}$ be the functions corresponding to $\phi_{11}$ and $\phi_{12}$ in $%
C^{0}[0,\sigma ]$ respectively with $
\max\limits_{t\in \left[ 0,\sigma \right]} \left| \phi_{ij}(t)\right| \leq
M,\,\quad i,j=1,2$. Under hypothesis Lemma \ref{cotasy} we have 
\begin{equation}
\left\{ 
\begin{array}{c}
\left| y_{01}(t)-y_{02}(t) \right| \leq
\frac{2D}{H}\sigma\left\| \phi_{11}-\phi_{12}\right\| _{\sigma } ,\\ 
\\ 
\left| y_{01}(t)-y_{02}(\tau) \right| \leq  \left(\beta+2\tfrac{DM}{H}\right)
\left| t-\tau \right| ,\text{ }i=1,2, \\ 
\\ 
\frac{C_{1}}{2}\leq y_{0i}(t) \leq \frac{3C_{1}}{2},\text{ }%
\forall t\in \left[ 0,\sigma \right] ,\text{ }i=1,2,
\end{array}
\right.
\end{equation}\label{desy}
and 
\begin{equation}
\left\{ 
\begin{array}{c}
\left| y_{11}(t)-y_{12}(t) \right| \leq
\frac{\beta +1}{\beta^{2}}\sigma \left\| \phi_{21}-\phi_{22}\right\| _{\sigma } ,\\ 
\\ 
\left| y_{1i}(t)-y_{1i}(\tau) \right| \leq
 (1+\beta)\left(1+\tfrac{M}{\beta^{2}}\right)\left| t-\tau \right| ,\text{ }i=1,2, \\ 
\\ 
\frac{C_{2}}{2}\leq y_{1i}(t) \leq \frac{3C_{2}}{2},%
\text{ }\forall t\in \left[ 0,\sigma \right] ,\text{ }i=1,2.
\end{array}
\right.
\end{equation}\label{desyi}
\end{lemma}

\begin{proof}
It follows inmediatly from definitions (\ref{ycero})-(\ref{ese}) and assumptions on data.
\end{proof}

\begin{lemma}\label{cotasintyi}
If we take $\sigma \leq 1$, $
\tfrac{4M}{H}\left(\beta + \tfrac{2DM}{H}\right)\sigma \leq 1$
and we assume the hypothesis of Lemma \ref{cotasyi} then we have 

\begin{equation}
\int_{C_{1}}^{C_{2}}\left| F^{^{\prime }}(\xi )\right| \left|
N(y_{11}(t),t;\xi ,0)-N(y_{12}(t),t;\xi ,0)\right| d\xi
\label{1iiprimero}
\end{equation}
\[
\leq \frac{2\left\|F^{\prime }\right\|_{[C_{1},C_{2}]}}{D\sqrt{\pi}}\left\|
\phi_{11}-\phi_{12}\right\|_{\sigma} \sqrt{\sigma } \leq P_{1}(u_{0},D,\beta,U_{0}) \left\| \vec{\phi_{1}^{*}}-\vec{\phi_{2}^{*}}\right\|\sqrt{\sigma},
\]
\begin{equation}
D\int_{0}^{t}\left| \phi_{11}(\tau )G_{y}(y_{11}(t),t;y_{11}(\tau
),\tau )-\phi_{12}(\tau )G_{y}(y_{12}(t),t;y_{12}(\tau ),\tau )\right|
d\tau  \label{1v}
\end{equation}
\[
\leq P_{2}(M,D,C_{2})\sigma \left\| \phi_{11}-\phi_{12}\right\|\leq P_{2}(M,D,C_{2})\ \left\| \vec{\phi_{1}^{*}}-\vec{\phi_{2}^{*}}\right\| \;\sqrt{\sigma} ,
\]

\begin{equation}
\beta^{2}\int_{0}^{t} \tfrac{h(\tau)}{f(\tau)}\left|G_{y}(y_{11}(t),t;y_{01}(\tau
),\tau )-G_{y}(y_{12}(t),t;y_{02}(\tau ),\tau
)\right| d\tau  \label{1iiii} 
\end{equation}
\[\leq  P_{3}(R,\beta,C_{1},C_{2})\left\| \vec{\phi_{1}^{*}}-\vec{\phi_{2}^{*}}\right\|\sqrt{\sigma} , \nonumber
\]
\begin{equation}
D\beta\int_{0}^{t} \tfrac{1}{f(\tau)}\left| \phi_{21}(\tau )G_{y}(y_{11}(t),t;y_{01}(\tau
),\tau )-\phi_{22}(\tau )G_{y}(y_{12}(t),t;y_{02}(\tau ),\tau
)\right| d\tau  \label{1iiiiii} 
\end{equation}
\[\leq P_{4}(D,M,C_{1},C_{2})\left\| \vec{\phi_{1}^{*}}-\vec{\phi_{2}^{*}}\right\|\sqrt{\sigma},  \nonumber
\]
\begin{equation}
\int_{0}^{t}\left|h'(\tau )\right| \left|
N(y_{11}(t),t;y_{01}(\tau ),\tau)-N(y_{12}(t),t,y_{02}(\tau ),\tau )\right| d\tau  \label{1iii}
\end{equation}
\[
\leq P_{5}(S,D,C_{1},C_{2})\left\| \vec{\phi_{1}^{*}}-\vec{\phi_{2}^{*}}\right\|\sqrt{\sigma} ,
\]

\begin{equation}
\int_{C_{1}}^{C_{2}}\left| F^{^{\prime }}(\xi )\right| \left|
N(y_{01}(t),t;\xi ,0)-N(y_{02}(t),t;\xi ,0)\right| d\xi
\label{1ii}
\end{equation}
\[
\leq \frac{2\left\|F^{\prime }\right\|_{[C_{1},C_{2}]}}{D\sqrt{\pi}}\left\|
\phi_{2,1}-\phi_{2,2}\right\|_{\sigma} \sqrt{\sigma } \leq P_{1}(u_{0},D,\beta,U_{0}) \left\|\vec{\phi_{1}^{*}}-\vec{\phi_{2}^{*}}\right\|\sqrt{\sigma},
\]
\begin{equation}
D\int_{0}^{t}\left| \phi_{11}(\tau )G_{y}(y_{01}(t),t;y_{11}(\tau
),\tau )-\phi_{12}(\tau )G_{y}(y_{02}(t),t;y_{12}(\tau ),\tau )\right|
d\tau  \label{1iibis}
\end{equation}
\[
\leq  P_{4}(D,M,C_{1},C_{2})\ \left\| \vec{\phi_{1}^{*}}-\vec{\phi_{2}^{*}}\right\| \;\sqrt{\sigma} ,
\]
\begin{equation}
\beta^{2}\int_{0}^{t}\frac{h(\tau)}{f(\tau)}\left|G_{y}(y_{01}(t),t;y_{01}(\tau
),\tau )-G_{y}(y_{02}(t),t;y_{02}(\tau ),\tau )\right|
d\tau  \label{1v}
\end{equation}
\[
\leq P_{6}(D,H,R,M,C_{1})\left\| \vec{\phi_{1}^{*}}-\vec{\phi_{2}^{*}}\right\| \;\sqrt{\sigma},
\]
\begin{equation}
D\int_{0}^{t}\left| \phi_{21}(\tau )G_{y}(y_{01}(t),t;y_{01}(\tau
),\tau )-\phi_{22}(\tau )G_{y}(y_{02}(t),t;y_{02}(\tau ),\tau )\right|
d\tau  \label{1vbis}
\end{equation}
\[
\leq P_{7}(M,D,C_{1})\ \left\| \vec{\phi_{1}^{*}}-\vec{\phi_{2}^{*}}\right\| \;\sqrt{\sigma},
\]

\begin{equation}
\int_{0}^{t}\left| h'(\tau )\right| \left|
N(y_{11}(t),t;y_{01}(\tau ),\tau
)-N(y_{12}(t),t,y_{0 2}(\tau ),\tau )\right| d\tau  \label{1vi}
\end{equation}
\[
\leq P_{5}(S,D,C_{1},C_{2})\left\| \vec{\phi_{1}^{*}}-\vec{\phi_{2}^{*}}\right\| \sqrt{\sigma} ,
\]

where 
\begin{equation}
P_{1}(u_{0},D,\beta,U_{0})=\tfrac{2}{D\sqrt{\pi}}\left[\left(\left\| u_{0}\right\|+\beta\right)exp\left(\tfrac{U_{0}}{\beta}\left(\left\| u_{0}\right\|+\beta\right)\right)+\tfrac{\left(\left\| u_{0}\right\|+\beta\right)^{2}}{D}\right],
\label{defp1}
\end{equation}
\begin{equation}
P_{2}(M,D,C_{2})=\tfrac{\sqrt{D}}{4\sqrt{\pi }}\left[ 6M+\tfrac{3}{C_{2}^{2}}\left( \tfrac{2%
}{3e}\right) ^{3/2}+\tfrac{6M}{C_{2}^{2}}\left( \tfrac{6}{e}\right) ^{3/2}\right],
\label{defp2}
\end{equation}
\begin{equation}
P_{3}(R,\beta,C_{1},C_{2})=R\beta(P_{31}+P_{32}),
\end{equation} with
\begin{equation}
P_{31}(C_{1},C_{2})=\tfrac{1}{\sqrt{\pi }e^{3/2}}\left[ \tfrac{\sqrt{6}\left(
3C_{2}-C_{1}\right) ^{2}}{16(C_{2}-3C_{1})^{3}}+\tfrac{27\sqrt{3}}{4}+\tfrac{12%
\sqrt{6}}{(C_{2}-3C_{1})^{3}}+\tfrac{6\sqrt{3}}{(C_{2}+C_{1})^{3}}\right],
\label{defp31}
\end{equation}
\begin{equation}
P_{32}(C_{1},C_{2})=\tfrac{12\sqrt{6}}{\sqrt{\pi }e^{3/2}}\left[ \tfrac{1}{%
(C_{2}-3C_{1})^{3}}+\tfrac{9}{8}+\tfrac{\left( 3C_{2}-C_{1}\right) ^{2}}{%
8(C_{2}-3C_{1})^{3}}+\tfrac{1}{(C_{2}+C_{1})^{2}}\right],  \label{defp32}
\end{equation}
\begin{equation}
P_{4}(D,M,C_{1},C_{2})=D\left[M(P_{31}+P_{32})+P_{41}\right],
\end{equation} where
\begin{equation}
P_{41}(C_{1},C_{2})=\tfrac{\sqrt{6}}{\sqrt{\pi e}}\left[ \tfrac{1}{%
(C_{2}-3C_{1})^{2}}+\tfrac{1}{(C_{2}+C_{1})^{2}}\right]  ,\label{defp3}
\end{equation}
\begin{equation}
P_{5}(S,D,C_{1},C_{2})=\tfrac{6^{3/2}SD}{\sqrt{\pi }e^{3/2}}\left[ \frac{%
3C_{2}-C_{1}}{(C_{2}-3C_{1})^{3}}+\frac{3}{(C_{2}+C_{1})^{2}}\right],
\label{defp5}
\end{equation}
\begin{equation}
P_{6}(D,H,R,M,C_{1})=
\beta R\left\lbrace(2D)^{-1} (D\pi)^{-1/2}\left[\tfrac{2D}{H}+\tfrac{2}{H}\left(\beta + \tfrac{2DM}{H}\right)^{2}\right]\right.
\end{equation}
\[ \left. +\left(\tfrac{6}{eC_{1}^{2}}\right)^{3/2}\tfrac{18C_{1}^{2}+1}{4\sqrt{\pi}}\tfrac{4D}{H}\right\rbrace,
\]

\begin{equation}
P_{7}(M,D,C_{1})=\tfrac{\sqrt{D}}{4\sqrt{\pi }}\left[ 6M+\tfrac{3}{C_{1}^{2}}\left( \tfrac{2%
}{3e}\right) ^{3/2}+\tfrac{6M}{C_{1}^{2}}\left( \tfrac{6}{e}\right) ^{3/2}\right].
\label{defp2}
\end{equation}
\end{lemma}
\begin{proof}
The inequalities (\ref{1iiprimero})-(\ref{1iibis}) and (\ref{1vbis})-(\ref{1vi}) are obtained following \cite{BrNa2012}.

We will show the proof of (\ref{1v}), following \cite{Sh}. We write 
\[
\left| G_{y}(y_{01}(t),t;y_{01}(\tau ),\tau
)-G_{y}(y_{02}(t),t;y_{02}(\tau ),\tau )\right| 
\]
\[
\leq \left|
K_{y}(y_{01}(t),t;y_{01}(\tau ),\tau )-K_{y}(y_{02}(t),t;y_{02}(\tau ),\tau )\right|\] 
\[+\left|
K_{y}(-y_{01}(t),t;y_{01}(\tau ),\tau )-K_{y}(-y_{02}(t),t;y_{02}(\tau ),\tau )\right|. 
\]
Taking into account that 
\[
\left|
K_{y}(y_{01}(t),t;y_{01}(\tau ),\tau )-K_{y}(y_{02}(t),t;y_{02}(\tau ),\tau )\right|\] 
\[
\leq(2D(t-\tau))^{-1} \left|K(y_{01}(t),t;y_{01}(\tau ),\tau )\left[\left(y_{01}(t)-y_{01}(\tau )\right)-\left(y_{02}(t)-y_{02}(\tau )\right)\right]\right. 
\]
\[
+\left.
\left[K(y_{01}(t),t;y_{01}(\tau ),\tau )-K(y_{02}(t),t;y_{02}(\tau ),\tau )\right]\left(y_{02}(t)-y_{02}(\tau )\right)\right|\] 
\[
\leq (2D(t-\tau))^{-1} K(y_{01}(t),t;y_{01}(\tau ),\tau )
\left|\left[\left(y_{01}(t)-y_{01}(\tau )\right)-\left(y_{02}(t)-y_{02}(\tau )\right)\right]\right.
 \]
\[
\left.+\left[1-exp(m(t,\tau)\right]\left(y_{02}(t)-y_{02}(\tau )\right)
\right|, 
\]
where
\[
m(t,\tau)=\frac{\left(y_{01}(t)-y_{01}(\tau )\right)^{2}-\left(y_{02}(t)-y_{02}(\tau )\right)^{2}}{4D(t-\tau)}
\]
\[=\frac{\left[\left(y_{01}(t)-y_{01}(\tau )\right)-\left(y_{02}(t)-y_{02}(\tau )\right)\right]\left[\left(y_{01}(t)-y_{01}(\tau )\right)+\left(y_{02}(t)-y_{02}(\tau )\right)\right]}{4D(t-\tau)}.
\]
We have 
\[
\left|\left(y_{01}(t)-y_{01}(\tau )\right)-\left(y_{02}(t)-y_{02}(\tau )\right)\right| \leq D\int_{\tau}^{t}\tfrac{1}{h(\eta)} \left| 1- \tfrac{\beta}{f(\eta)}\right|\left|\phi_{21}(\eta)-\phi_{22}(\eta)\right| d\eta 
\]
\[
\leq \tfrac{2D}{H}\left\| \vec{\phi_{1}^{*}}-\vec{\phi_{2}^{*}}\right\|(t-\tau),
\]
and 
\[
\left|\left(y_{01}(t)-y_{01}(\tau )\right)+\left(y_{02}(t)-y_{02}(\tau )\right)\right| \leq 2\left(\beta + \tfrac{2DM}{H}\right)(t-\tau)\leq 4\left(\beta + \tfrac{2DM}{H}\right)\sigma,
\]
then
\[
\left| m(t,\tau)\right|\leq\tfrac{2}{H}\left\| \vec{\phi_{1}^{*}}-\vec{\phi_{2}^{*}}\right\|\left(\beta + \tfrac{2DM}{H}\right)\sigma,
\]
and taking into account that $\left\| \vec{\phi_{1}^{*}}-\vec{\phi_{2}^{*}}\right\|\leq 2M$ we have
\[
\left| m(t,\tau)\right|\leq\tfrac{4M}{H}\left(\beta + \tfrac{2DM}{H}\right)\sigma.
\]
If we assume that $\sigma$ satisfies \[\tfrac{4M}{H}\left(\beta + \tfrac{2DM}{H}\right)\sigma \leq 1,
\]
we obtain that 
\[\left|1-exp(m(t,\tau)\right|\leq 2\left| m(t,\tau)\right|\leq \tfrac{2}{H}\left(\beta + \tfrac{2DM}{H}\right)\left\| \vec{\phi_{1}^{*}}-\vec{\phi_{2}^{*}}\right\|\sigma.
\]
Therefore
\[
\left|
K_{y}(y_{01}(t),t;y_{01}(\tau ),\tau )-K_{y}(y_{02}(t),t;y_{02}(\tau ),\tau )\right|\leq 
\] 
\[
\leq (2D)^{-1} K(y_{01}(t),t;y_{01}(\tau ),\tau )
\left[\tfrac{2D}{H}+\tfrac{2}{H}\left(\beta + \tfrac{2DM}{H}\right)^{2}\sigma\right]\left\| \vec{\phi_{1}^{*}}-\vec{\phi_{2}^{*}}\right\|
 \]
\[
\leq (4D)^{-1} (D\pi(t-\tau))^{-1/2}
\left[\tfrac{2D}{H}+\tfrac{2}{H}\left(\beta + \tfrac{2DM}{H}\right)^{2}\sigma\right]\left\| \vec{\phi_{1}^{*}}-\vec{\phi_{2}^{*}}\right\|.
 \]

Using the mean value theorem we may write  
\[\left|
K_{y}(-y_{01}(t),t;y_{01}(\tau ),\tau )-K_{y}(-y_{02}(t),t;y_{02}(\tau ),\tau )\right| 
\]
\[\leq \left|
K(n(t,\tau),t;0,\tau)\left( \frac{n^{2}(t,\tau)}{4D^{2}(t-\tau)^{2}}-\frac{1}{2D(t-\tau)}\right)\right| \left|y_{01}(t)+y_{01}(\tau)-y_{02}(t)- y_{02}(\tau)\right|
\]
where $n=n\left(
t,\tau\right)$ is between $y_{01}(t)+y_{01}(\tau)\;$ and $y_{02}(t)+ y_{02}(\tau)$.

Since
\[ \left|y_{01}(t)+y_{01}(\tau)-y_{02}(t)- y_{02}(\tau)\right|\leq \frac{4D}{H}\sigma \left\| \vec{\phi_{1}^{*}}-\vec{\phi_{2}^{*}}\right\|,
\]
and
\[C_{1}\leq n\left(
t,\tau\right)\leq 6 C_{1},\]
by using $(\ref{exp})$ we have
\[\left|
K(n(t,\tau),t;0,\tau)\left( \frac{n^{2}(t,\tau)}{4D^{2}(t-\tau)^{2}}-\frac{1}{2D(t-\tau)}\right)\right|\leq\left(\frac{6}{eC_{1}^{2}}\right)^{3/2}\frac{18C_{1}^{2}+1}{4\sqrt{\pi}},
\]
then
\[\left|
K_{y}(-y_{01}(t),t;y_{01}(\tau ),\tau )-K_{y}(-y_{02}(t),t;y_{02}(\tau ),\tau )\right| \leq \left(\tfrac{6}{eC_{1}^{2}}\right)^{3/2}\tfrac{18C_{1}^{2}+1}{4\sqrt{\pi}}\tfrac{4D}{H}\sigma \left\| \vec{\phi_{1}^{*}}-\vec{\phi_{2}^{*}}\right\|.
\]
Collecting the results we have
\[
\left| G_{y}(y_{01}(t),t;y_{01}(\tau ),\tau
)-G_{y}(y_{02}(t),t;y_{02}(\tau ),\tau )\right| 
\]
\[
\leq \left\lbrace(4D)^{-1} (D\pi(t-\tau))^{-1/2}
\left[\tfrac{2D}{H}+\tfrac{2}{H}\left(\beta + \tfrac{2DM}{H}\right)^{2}\right]+\left(\tfrac{6}{eC_{1}^{2}}\right)^{3/2}\tfrac{18C_{1}^{2}+1}{4\sqrt{\pi}}\tfrac{4D}{H}\right\rbrace \left\| \vec{\phi_{1}^{*}}-\vec{\phi_{2}^{*}}\right\|,
\]
then
\[
\beta^{2}\int_{0}^{t}\frac{h(\tau)}{f(\tau)}\left|G_{y}(y_{01}(t),t;y_{01}(\tau
),\tau )-G_{y}(y_{02}(t),t;y_{02}(\tau ),\tau )\right|d\tau 
\]
\[\leq P_{6}(D,H,R,M,C_{1}\left\| \vec{\phi_{1}^{*}}-\vec{\phi_{2}^{*}}\right\|\sqrt{\sigma},
 \]
where 
\[P_{6}(D,H,R,M,C_{1})=
\beta R\left\lbrace(2D)^{-1} (D\pi)^{-1/2}\left[\tfrac{2D}{H}+\tfrac{2}{H}\left(\beta + \tfrac{2DM}{H}\right)^{2}\right]\right.
\]
\[ \left. +\left(\tfrac{6}{eC_{1}^{2}}\right)^{3/2}\tfrac{18C_{1}^{2}+1}{4\sqrt{\pi}}\tfrac{4D}{H}\right\rbrace.
\]

Then, (\ref{1v}) has been proved

\end{proof}

\begin{theorem}\label{teocont} Let hypothesis $(\ref{hip})$ be. Fixed $0<C_{1}<\frac{U_{0}}{2}$ and $h\in \Pi$. If $\sigma $ satisfies the following inequalities 

\begin{equation}
\sigma \leq 1,\,\quad 2(1+\beta)\left(1+\frac{M}{\beta^{2}}\right)\sigma \leq C_{2},\quad\quad \; \label{nose}
\end{equation}
\begin{equation}
\left(\beta+2\frac{MD}{H}\right)\sigma\leq C_{1},\quad\quad\tfrac{4M}{H}\left(\beta + \tfrac{2DM}{H}\right)\sigma \leq 1,
\end{equation}
\begin{equation}
H_{1}\left( C_{1},C_{2},U_{0},M,D,\beta,R,S,\sigma  \right) \leq 1,
\label{ache1}
\end{equation}
\begin{equation}
H_{2}(C_{1},C_{2},U_{0},M,D,\beta,R,S,\sigma  )\leq 1,  \label{ache2}
\end{equation}
where $M$ is given by 
\begin{equation}
M\left(u_{0,}U_{0},f,D,\beta,R\right)=1+\left(\frac{1}{2-D}+\frac{\|f\|}{\beta(3-D)}\right)2A_{1}+2\frac{\|f\|}{3-D}R, \label{erre}
\end{equation}
and 
\[
H_{1}\left( C_{1},U_{0},f,M,D,\beta,R,S,\sigma\right) =\left\lbrace\left(\frac{2}{2-D}\right) \left[A_{2}+A_{3}+A_{4}+\frac{2 S}{\sqrt{\pi D}}\right]\right.
\]
\begin{equation}
\left. +\frac{2 \|f\|}{\beta (3-D)}\left[A_{4}+A_{5}+A_{6}+ \frac{2 S}{\sqrt{\pi D}}\right]\right\rbrace\sqrt{\sigma} ,\label{defache1}
\end{equation}

\[
H_{2}\left(C_{1},U_{0},f,M,D,\beta,R,S,\sigma  \right) = \left\lbrace\frac{2}{2-D}\left[P_{1}+P_{2}+P_{3}+P_{4}+P_{5}\right]\right.
\]
\begin{equation}
\left. +\frac{2 \|f\|}{\beta (3-D)}\left[P_{1}+P_{4}+P_{5}+P_{6} +P_{7}\right]\right\rbrace\sqrt{\sigma}, \label{defache2}
\end{equation}
then the map $\chi:C_{M,\sigma }\longrightarrow C_{M,\sigma }$ is well defined and it is a contraction map. Therefore there exists a unique solution $\phi_{1}^{*}$, $\phi_{2}^{*}$ on $C_{M,\sigma }$ to the system of
integral equations (\ref{ecintegralf}) and (\ref{ecintegralf1}).
\end{theorem}

\begin{proof}
Firstly, we demonstrate that $\chi
$ maps $C_{M,\sigma }\;$into itself, that is 
\[
\left\| X\left( \stackrel{\longrightarrow }{\phi^{*}}\right) \right\| _{\sigma
}=\max \limits_{t\in \left[ 0,\sigma \right] }\left|
\chi_{1}(\phi_{1}(t),\phi_{2}(t))\right| +\max \limits_{t\in \left[ 0,\sigma \right] }
\left| \chi_{2}(\phi_{1}(t),\phi_{2}(t))\right| \leq M 
\]
Taking into account Lemma \ref{cotasint} we have 
\[
\left| \chi_{1}(\phi_{1}(t),\phi_{2}(t))\right| \leq \frac{2}{2-D} \left\lbrace A_{1}+\left[A_{2}+A_{3}+A_{4}+\frac{2 S}{\sqrt{\pi D}}\right]\sqrt{\sigma}\right\rbrace,
\]
\[
\left| \chi_{2}(\phi_{1}(t),\phi_{2}(t))\right| \leq \frac{2 \|f\|}{\beta (3-D)}\left\lbrace\beta R+A_{1}+\left[A_{4}+A_{5}+A_{6}+ \frac{2 S}{\sqrt{\pi D}}\right]\sqrt{\sigma}\right\rbrace,
\]
and then 
\[
\left\| \chi\left( \stackrel{\longrightarrow }{\phi^{*}}\right) \right\| _{\sigma
}\leq   2A_{1}\left[\frac{1}{2-D}+\frac{\|f\|}{\beta (3-D)}\right] +\frac{2 \|f\| R}{(3-D)}+H_{1}(C_{1},C_{2},U_{0},M,D,\beta,R,S,\sigma)
)\]
where $H_{1}$ is given by $\left( \ref{defache1}\right).$ 

Selecting $M$ by $%
\left( \ref{erre}\right) \;$and $\sigma \;$such that $\left( \ref{ache1}%
\right) \;$holds, we obtain $\left\| \chi\left( \stackrel{\longrightarrow }{\phi^{*}%
}\right) \right\| _{\sigma }\leq M.$ 

Now, we will prove that 
\[
\left\| \chi\left( \stackrel{\longrightarrow }{\phi_{1}^{*}}\right) -\chi\left( 
\stackrel{\longrightarrow }{\phi_{2}^{*}}\right) \right\| _{\sigma }\leq
H_{2}\left( C_{1},C_{2},U_{0},M,D,\beta,R,S,\sigma \right) \left\| 
\stackrel{\longrightarrow }{\phi_{1}^{*}}-\stackrel{\longrightarrow }{\phi_{2}^{*}}%
\right\| _{\sigma } 
\]
where $\stackrel{\longrightarrow }{\phi_{1}^{*}}=\binom{\phi_{11}}{\phi_{12}}\;,\;%
\stackrel{\longrightarrow }{\phi_{2}^{*}}=\binom{\phi_{21}}{\phi_{22}}$ $\in
C_{M,\sigma }$. 

Taking into account Lemma \ref{cotasintyi} we have 
\[
\left\| \chi\left( \stackrel{\longrightarrow }{\phi_{1}^{*}}\right) -\chi\left( 
\stackrel{\longrightarrow }{\phi_{2}^{*}}\right) \right\| _{\sigma }=
\max\limits_{t\in \left[ 0,\sigma \right] }\left| \chi_{1}\left(
\phi_{11}\left( t\right) ,\phi_{12}\left( t\right) \right) -\chi_{1}\left( \phi_{21}\left(
t\right) ,\phi_{22}\left( t\right) \right) \right| 
\]
\[
+\max\limits_{t\in \left[ 0,\sigma \right] }\left| \chi_{2}\left(
\phi_{11}\left( t\right) ,\phi_{12}\left( t\right) \right) -\chi_{2}\left( \phi_{21}\left(
t\right) ,\phi_{22}\left( t\right) \right) \right| 
\]

\[
\leq\left\lbrace\frac{2}{2-D}\left[P_{1}+P_{2}+P_{3}+P_{4}+P_{5}\right]\right.
\]
\begin{equation}
\left. +\frac{2 \|f\|}{\beta (3-D)}\left[P_{1}+P_{4}+P_{5}+P_{6} +P_{7}\right]\right\rbrace\sqrt{\sigma}\left\| 
\stackrel{\longrightarrow }{\phi_{2}^{*}}-\stackrel{\longrightarrow }{\phi_{1}^{*}}%
\right\| _{\sigma } 
\end{equation}
\[
= H_{2}\left( C_{1},C_{2},U_{0},M,D,\beta,R,S,\sigma \right) \left\| 
\stackrel{\longrightarrow }{\phi_{2}^{*}}-\stackrel{\longrightarrow }{\phi_{1}^{*}}%
\right\| _{\sigma }.  
\]
By hypothesis (\ref{nose})-(\ref{erre}) we have that $\chi$ is a contraction and therefore, there exists a unique fixed point $\phi^{*}=\binom{\phi_{1}}{\phi_{2}}$ such that $\chi(\phi^{*})=\phi^{*}$ this is 
$$\chi_{1}(\phi_{1}(t),\phi_{2}(t))=\phi_{1}(t),\quad\quad \chi_{2}(\phi_{1}(t),\phi_{2}(t))=\phi_{2}(t). $$
\end{proof}
\begin{theorem} For each $h\in\Pi_{1}$, under hypothesis of Theorem \ref{teocont} there exists a unique integral representation for $w$, $y_{0}$ and $y_{1}$ given by $(\ref{z})$,$(\ref{ycero})$ and $(\ref{ese})$ respectively, where   $\phi_{1}$ and $\phi_{2}$ are the unique solutions of (\ref{ecintegralf}) and (\ref{ecintegralf1})
\end{theorem}
\end{subsection}

\subsection{Existence of at least a solution of $w_{h}(y_{0h}(t),t)=h(t)$}

In this subsection we assume that all the hypothesis of Theorem \ref{teocont} are valid, which guarantee the existence and uniqueness of $w=w_{h}$, $y_{0}=y_{0h}$ and $y_{1}=y_{1h}$ for each $h\in\Pi_{1}$.

Now we will prove that for suitable values of $H,R,S$ and $\sigma$ there exists $h\in\Pi(H,R,S,\sigma)$ such that 
\begin{equation}
w_{h}(y_{0h}(t),t)=h(t)\label{ecuache}
\end{equation}
 for $t\in [0,\sigma]$, where $w_{h}$ and $y_{0h}$ are established for above theorem.

We define the map $Z$ on $\Pi$ such that for each $h\in\Pi\subset \Pi_{1}$ $(\sigma\leq1)$
$$Z(h)(t)=w_{h}(y_{0h}(t),t),$$
this is
\begin{equation}
Z(h)(t)=(f(t)-\beta)\left(1-\int_{0}^{t}\phi_{1h}(\tau) d\tau
+\frac{1}{D}\int_{y_{0h}(t)}^{y_{1h}(t)}w_{h}(\xi,t)d\xi\right),\label{ache1}
\end{equation}
where $\phi_{1h}$, $w_{h}$, $y_{0h}$ and $y_{1h}$ are the solutions obtained in above section.
We will use the Schauder's fixed point theorem which states: \textit{For any continuous function $L$ mapping a compact convex set to itself there is $x_{0}$ such that $L(x_{0})=x_{0}$ }. 

\begin{lemma}\label{propZh}If 
\begin{equation}
\left(\left\|f\right\|+\beta\right)(2C_{1}+3U_{0})<2D\label{ipp}
\end{equation} then, for $h\in \Pi$ function $Z(h)\in C^{1}[0,\sigma]$ and satisfies
\begin{equation}
Z(h)(t)>\frac{\beta}{2} ,\label{H}
\end{equation}
\begin{equation}
\left\|Z(h)\right\|\leq \frac{2D\left(\left\|f\right\|+\beta\right)(1+M)}{2D-\left(\left\|f\right\|+\beta\right)(2C_{1}+3U_{0})},\label{RR}
\end{equation}
\begin{equation}
\left\|Z'(h)\right\|\leq  \tfrac{2D\left(\left\|f\right\|+\beta\right)(1+M)}{2D-\left(\left\|f\right\|+\beta\right)(2C_{1}+3U_{0})}\left\lbrace\tfrac{2\left\|f'\right\|}{\beta}
+\left(\left\|f\right\|+\beta\right) \left[\tfrac{\beta}{D}+M\left(\tfrac{2}{\beta}+\tfrac{1}{D}\right)\right]\right\rbrace\label{S}
\end{equation}
\[
+M\left(\left\|f\right\|+\beta\right).\]
\end{lemma}
\begin{proof} From definition of $Z(h)$ and $(\ref{hip})$ we have $Z(h)(t)>\frac{\beta}{2}$. Taking into account $(\ref{ache1})$ and the fact that
\[ w_{h}(\xi,t) \leq  w_{h}(y_{0}(t),t),\quad y_{0}(t)\leq \xi \leq y_{1}(t)
\] we have 
\[
Z(h)(t)\leq (f(t)+\beta)\left(1-\int_{0}^{t}\phi_{1h}(\tau) d\tau)+w_{h}(y_{0}(t),t)\left(y_{1}(t)-y_{0}(t)\right)\right).
\]
Since 
$\left\|\phi_{1h}\right\|\leq M$, $\sigma\leq 1$ and taking into account Lemma $\ref{cotasyi}$ we obtain
\[\left\|Z(h)\right\|
\leq\left(\left\|f\right\|+\beta\right)\left(1+M+\left\|Z(h)\right\|\frac{2C_{1}+3U_{0}}{2D}\right),
\]
or equivalently
\[
\left\|Z(h)\right\|\leq \frac{2D\left(\left\|f\right\|+\beta\right)(1+M)}{2D-\left(\left\|f\right\|+\beta\right)(2C_{1}+3U_{0})}.
\]
If we derivate $Z(h)$ respect to variable $t$, we get
\[
(Z(h))'(t)=f'(t)\left(1-\int_{0}^{t}\phi_{1h}(\tau) d\tau
+\tfrac{1}{D}\int_{y_{0h}(t)}^{y_{1h}(t)}w_{h}(\xi,t)d\xi\right)\]
\[+(f(t)-\beta)\left[-\phi_{1h}(t)+\tfrac{1}{D} w_{h}(y_{0h}(t),t) y'_{0h}(t)+\frac{1}{D}\int_{y_{0h}(t)}^{y_{1h}(t)} w_{ht}(\xi,t)d\xi\right]\,
\]
and using eq. $(\ref{calu} )$, $(\ref{ycero})$ and $(\ref{def})$ we obtain
\[
(Z(h))'(t)=f'(t)\left(1-\int_{0}^{t}\phi_{1h}(\tau) d\tau
+\tfrac{1}{D}\int_{y_{0h}(t)}^{y_{1h}(t)}w_{h}(\xi,t)d\xi\right)\]
\[+(f(t)-\beta)\left[-\phi_{1h}(t)-\tfrac{1}{D} w_{h}(y_{0h}(t),t) y'_{0h}(t)+\phi_{h1}(t)-\phi_{h2}(t)\right]
\]
\[
=f'(t)\frac{Z(h)(t)}{f(t)-\beta}+(f(t)-\beta)\left[\tfrac{Z(h)\beta^{2}}{Df(t)} +\phi_{h2}(t)\left(\tfrac{Z(h)(t)}{h(t)}-1-\tfrac{Z(h)(t)\beta}{Df(t)}\right)\right].
\]
Then we have 
\[
\left|(Z(h))'(t)\right|\leq \left\|f'\right\|\tfrac{2\left\|Z(h)\right\|}{\beta}
+\left(\left\|f\right\|+\beta\right)\left[\tfrac{\left\|Z(h)\right\|\beta}{D}+M\left(\tfrac{2\left\|Z(h)\right\|}{\beta}+1+\tfrac{\left\|Z(h)\right\|}{D}\right)\right]\]
\[
\leq\left\|Z(h)\right\|\left\lbrace\tfrac{2\left\|f'\right\|}{\beta}
+\left(\left\|f\right\|+\beta\right) \left[\tfrac{\beta}{D}+M\left(\tfrac{2}{\beta}+\tfrac{1}{D}\right)\right]\right\rbrace+M\left(\left\|f\right\|+\beta\right),
\]

\[
\leq  \tfrac{2D\left(\left\|f\right\|+\beta\right)(1+M)}{2D-\left(\left\|f\right\|+\beta\right)(2C_{1}+3U_{0})}\left\lbrace\tfrac{2\left\|f'\right\|}{\beta}
+\left(\left\|f\right\|+\beta\right) \left[\tfrac{\beta}{D}+M\left(\tfrac{2}{\beta}+\tfrac{1}{D}\right)\right]\right\rbrace+M\left(\left\|f\right\|+\beta\right)
\]
and the lemma holds.
\end{proof}
Next, we define
\begin{equation}
E_{1}=1+2\left(\frac{1}{2-D}+\frac{\|f\|}{\beta (3-D)}\right)A_{1},\quad\quad E_{2}=2\frac{\|f\|}{(3-D)}, 
\end{equation}and
$$E_{3}=\frac{2D\left(\left\|f\right\|+\beta\right)}{2D-\left(\left\|f\right\|+\beta\right)(2C_{1}+3U_{0})}.$$

\begin{lemma}\label{ZhenPi}
We assume $(\ref{ipp})$ and
\begin{equation}E_{3}E_{2}<1.\label{ee}
\end{equation}
If we take
\begin{equation}
H=\frac{\beta}{2},\quad\quad R=\frac{E_{3}(1+E_{1})}{1-E_{3}E_{2}}\quad,\label{defHR}
\end{equation}
\begin{equation}
S= E_{3}\left\lbrace\tfrac{2\left\|f'\right\|}{\beta}
+\left(\left\|f\right\|+\beta\right) \left[\tfrac{\beta}{D}+M\left(\tfrac{2}{\beta}+\tfrac{1}{D}\right)\right]\right\rbrace
\label{defS}\end{equation}
where $M$ is given by 
\begin{equation}
M=\frac{E_{1}+E_{2}E_{3}}{1-E_{3}E_{2}}\label{emee}
\end{equation}
then $Z_{h}\in\Pi.$
\end{lemma}
\begin{proof} 
From $(\ref{erre})$ we have 
\[
M=E_{1}+E_{2}R,\] 
then by $(\ref{RR})$ we have

\[E_{3}(1+M)=R \quad \Leftrightarrow\quad E_{3}(1+E_{1})+E_{3}E_{2}R=R.\]
Therefore, if we define $R=\frac{E_{3}(1+E_{1})}{1-E_{3}E_{2}}$  we have $
\left\|Z(h)\right\|\leq R.$ Moreover we have $Z(h)(t)>\frac{\beta}{2}=H$ and by Lemma \ref{propZh} we have $\left\|(Z(h))'\right\|\leq S.$ This yields $Z(h)\in\Pi$ and the proof is complete.

\end{proof}
\begin{remark}
Assumption  $(\ref{ee})$ is equivalent to 
\begin{equation}
\frac{4D\|f\|\left(\left\|f\right\|+\beta\right)}{\left[2D-\left(\left\|f\right\|+\beta\right)(2C_{1}+3U_{0})\right](3-D)}<1
\label{ip}
\end{equation}
\end{remark} 

\begin{theorem} We assume hypothesis of Lemma $\ref{ZhenPi}$.
There exists at least a solution $h^{*}\in\Pi$ such that $Z(h^{*})=h^{*}$.
\end{theorem}
\begin{proof}
Taking into account above lemmas and using Schauder's fixed-point theorem we obtain that there exists at least a solution $h^{*}\in\Pi$ such that $Z(h^{*})=h^{*}.$
\end{proof}
We can now formulate our main result.

\begin{theorem} \label{teoremafinal}Fixed $C_{1}<\frac{U_{0}}{2}$. Let $H$, $R$, $S$ and $M$ given by $(\ref{defHR})$, $(\ref{defS})$ and $(\ref{emee})$ respectively.
If $(\ref{ipp})$ and $(\ref{ip})$ hold,\begin{equation}
\sigma \leq 1,\,\quad 2(1+\beta)\left(1+\frac{M}{\beta^{2}}\right)\sigma \leq C_{2},\quad\quad \; \label{nose3}
\end{equation}
\begin{equation}
\left(\beta+2\frac{MD}{H}\right)\sigma\leq C_{1}\quad\quad\tfrac{4M}{H}\left(\beta + \tfrac{2DM}{H}\right)\sigma \leq 1
\end{equation}
\begin{equation}
H_{1} \leq 1, \quad\quad H_{2}\leq 1  \label{ache23}
\end{equation}
where $H_{1}$and $H_{2}$ are given by $(\ref{defache1})-(\ref{defache2})$ then 
there exists solution to the free boundary problem $(\ref{calu})-(\ref{free})$ given by
\begin{equation}
w^{*}(y,t)=\int\nolimits_{C_{1}}^{C_{2}}G(y,t;\xi ,0)F(\xi
)d\xi +D
\int_{0}^{t}\phi^{*}_{1}(\tau )G(y,t;y^{*}_{1}(\tau ),\tau )d\tau  \label{1z+}
\end{equation}
\[
+\beta^{2}\int_{0}^{t} \frac{h^{*}(\tau)}{f(\tau)} G(y,t;y^{*}_{0}(\tau ),\tau )d\tau
-D\beta\int_{0}^{t} \frac{\phi^{*}_{2}(\tau)}{f(\tau)} G(y,t;y^{*}_{0}(\tau ),\tau ) d\tau 
\]
\[-D\int_{0}^{t} h^{*}(\tau) N_{y}(y,t;y^{*}_{0}(\tau ),\tau ) d\tau 
\]
and 
\begin{equation}
y^{*}_{0}(t)=C_{1}-\beta^{2}\int^{t}_{0} \frac{1}{f(\tau)}d\tau-D\int^{t}_{0}\tfrac{1}{h^{*}(\tau)}\left(1-\tfrac{\beta}{f(\tau)}\right)\phi^{*}_{2}(\tau)d\tau\label{ycero11}
\end{equation}
\begin{equation}
y^{*}_{1}(t)=C_{2}+(1-\beta)t+\tfrac{D(\beta +1)}{\beta^{2}}ln\left(1-\int^{t}_{0} \phi^{*}_{1}(\tau)d\tau\right) \label{ese11}
\end{equation}
where $\phi_{1}^{*}$, $\phi_{2}^{*}$ are the unique solutions to the system of two Volterra integral equations (\ref
{ecintegralf})\ and (\ref{ecintegralf1}) corresponding to $h^{*}$ a solution of $Z(h)=h.$

\end{theorem}

\section{Parametric solution to the problem (\ref{calor1})-(\ref{tempborde}) }

Assuming hypothesis of Theorem $\ref{teoremafinal}$, if we invert the transformations given by $(\ref{firsttrans})$, $(\ref{2})$ and $(\ref{tres})$ we obtain the explicit parametric representation of the solution to free boundary problem $(\ref{calor1})-(\ref{tempborde})$ given by
\begin{equation}
u^{*}(x,t)=\frac{w^{*}(y,t)}{1-\int_{0}^{t}\phi^{*}_{1}(\tau) d\tau+\frac{1}{D}\int_{y}^{y^{*}_{1}(t)} w^{*}(\xi,t)d\xi}+\beta \quad
\end{equation}
\begin{equation}
x=\int_{y^{*}_{0}(t)+2\beta t}^{y+2\beta t} \left[\frac{w^{*}(\mu,t)}{1-\int_{0}^{t}\phi^{*}_{1}(\tau) d\tau+\frac{1}{D}\int_{\mu}^{y^{*}_{1}(t)} w^{*}(\xi,t)d\xi}+\beta\right] d\mu
\end{equation}
with
\[
y^{*}_{0}(t)<y<y^{*}_{1}(t),\quad\quad 0<t<\sigma
\]
and
\begin{equation}
s(t)=\int_{y^{*}_{0}(t)+2\beta t}^{y^{*}_{1}(t)+2\beta t} \left[\frac{w^{*}(\mu,t)}{1-\int_{0}^{t}\phi^{*}_{1}(\tau) d\tau+\frac{1}{D}\int_{\mu}^{y^{*}_{1}(t)} w^{*}(\xi,t)d\xi}+\beta\right] d\mu
\end{equation}
where  $w^{*}=w^{*}(y,t)$ is given by $(\ref{1z+})$, $y^{*}_{0}(t)$, $y^{*}_{1}(t)$ are given by $(\ref{ycero11})$ and $(\ref{ese11})$ respectively, with $\phi_{1}^{*}$, $\phi_{2}^{*}$ the unique solutions to  (\ref{ecintegralf})\ and (\ref{ecintegralf1}) corresponding to the  solution $h^{*}$ to equation $(\ref{ecuache}).$
\end{section}

\section*{ACKNOWLEDGEMENT}
The present work has been partially sponsored by the Project PIP No 0275 from CONICET-UA, Rosario, Argentina, ANPCyT PICTO Austral 2016 No 0090 and 
the European Union's Horizon 2020 Research and
Innovation Programme under the Marie Sklodowska-Curie grant agreement 823731 CONMECH. The authors would like to thanks the anonymous referees whose insighful comments have benefited the presentation of this article.

\end{document}